\documentclass[letterpaper,11pt]{article}
 
\usepackage{etoolbox}
\newtoggle{anonymous}
\togglefalse{anonymous}

\usepackage[usenames,dvipsnames,svgnames,table]{xcolor}
\usepackage[colorlinks=true,linkcolor=blue,citecolor=ForestGreen]{hyperref}
\usepackage{amsmath,amssymb,amsfonts,amsthm}
\usepackage{enumerate}
\usepackage{scrextend}
\usepackage[utf8]{inputenc}
\usepackage[T1]{fontenc}
\usepackage[capitalize]{cleveref}
\usepackage{paralist}
\usepackage{bm}
\usepackage{thmtools}
\usepackage{float}
\usepackage[color=green!40,textsize=small]{todonotes}
\usepackage{orcidlink}
  \usepackage[margin=1in]{geometry}
	
	\usepackage{setspace}
\newtheorem{innercustomthm}{Theorem}
\newenvironment{customthm}[1]
  {\renewcommand\theinnercustomthm{#1}\innercustomthm}
  {\endinnercustomthm}

\newtheorem{innercustomcor}{Corollary}
\newenvironment{customcor}[1]
  {\renewcommand\theinnercustomcor{#1}\innercustomcor}
  {\endinnercustomcor}

\Crefname{figure}{Figure}{Figures}
\Crefname{claim}{Claim}{Claims}
\Crefname{theorem}{Theorem}{Theorems}

\crefformat{equation}{\textup{#2(#1)#3}}
\crefrangeformat{equation}{\textup{#3(#1)#4--#5(#2)#6}}
\crefmultiformat{equation}{\textup{#2(#1)#3}}{ and \textup{#2(#1)#3}}
{, \textup{#2(#1)#3}}{, and \textup{#2(#1)#3}}
\crefrangemultiformat{equation}{\textup{#3(#1)#4--#5(#2)#6}}%
{ and \textup{#3(#1)#4--#5(#2)#6}}{, \textup{#3(#1)#4--#5(#2)#6}}{, and \textup{#3(#1)#4--#5(#2)#6}}

\Crefformat{equation}{Equation~\textup{#2(#1)#3}}
\Crefrangeformat{equation}{Equations~\textup{#3(#1)#4--#5(#2)#6}}
\Crefmultiformat{equation}{Equations~\textup{#2(#1)#3}}{ and \textup{#2(#1)#3}}
{, \textup{#2(#1)#3}}{, and \textup{#2(#1)#3}}
\Crefrangemultiformat{equation}{Equations~\textup{#3(#1)#4--#5(#2)#6}}%
{ and \textup{#3(#1)#4--#5(#2)#6}}{, \textup{#3(#1)#4--#5(#2)#6}}{, and \textup{#3(#1)#4--#5(#2)#6}}

\crefname{section}{Section}{Sections}
\crefname{theorem}{Theorem}{Theorems}
\crefname{claim}{Claim}{Claims}
\crefname{lemma}{Lemma}{Lemmas}

\usepackage{tikz}
\tikzset{normalnode/.style={circle, draw, fill=black, inner sep=0, minimum width=1.5mm}}

\usetikzlibrary{fit}
\usetikzlibrary{shapes,arrows,shadows,positioning}
\usetikzlibrary{backgrounds}

\newcommand{\Nn}{\mathbb{N}}

\newcommand{\Fc}{\mathcal{F}}
\newcommand{\Mc}{\mathcal{M}}

\newcommand{\Xc}{\mathcal{X}}
\newcommand{\Yc}{\mathcal{Y}}
\newcommand{\Zc}{\mathcal{Z}}

\newcommand{\Sc}{\mathcal{S}}
\newcommand{\Cc}{\mathcal{C}}

\newtheorem{theorem}{Theorem}[section]
\newtheorem{question}{Question}

\newtheorem{corollary}[theorem]{Corollary}

\newtheorem{lemma}[theorem]{Lemma}

\theoremstyle{definition}
\newtheorem{definition}[theorem]{Definition}

\theoremstyle{remark}
\newtheorem{remark}[theorem]{Remark}
\newtheorem{claim}[theorem]{Claim}

\newenvironment{pocd}[1]{\begin{proof}[Proof of {C}laim #1.]}{\end{proof}}

\newcommand{\her}{\mathrm{Her}}
\newcommand{\mon}{\mathrm{Mon}}

\renewcommand{\leq}{\leqslant}
\renewcommand{\geq}{\geqslant}

\renewcommand{\Pr}[1]{\mathbb{P}\left[\,#1\,\right]}
\newcommand{\Ex}[1]{\mathbb{E} \left[\,#1\,\right]}
 
\newcommand{\comment}[1]{}

\begin{document}

\title{Small But Unwieldy:\\ A Lower Bound on Adjacency Labels for Small Classes\thanks{A conference version of this paper will appear at SODA 2024 \cite{BDSZZ23}.}}

\author{{\'E}douard Bonnet\thanks{Univ. Lyon, ENS de Lyon, UCBL, CNRS, LIP, France,  \texttt{edouard.bonnet@ens-lyon.fr}, \orcidlink{0000-0002-1653-5822}} \and Julien Duron\thanks{Univ. Lyon, ENS de Lyon, UCBL, CNRS, LIP, France, \texttt{julien.duron@ens-lyon.fr}, \orcidlink{0009-0004-0925-9438}} \and    John Sylvester\thanks{Department of Computer Science, University of Liverpool, UK, \texttt{john.sylvester@liverpool.ac.uk}, \orcidlink{0000-0002-6543-2934}}
		\and
		Viktor Zamaraev\thanks{Department of Computer Science, University of Liverpool, UK, \texttt{viktor.zamaraev@liverpool.ac.uk}, \orcidlink{0000-0001-5755-4141}}
	\and Maksim Zhukovskii\thanks{Department of Computer Science, University of Sheffield, UK, \texttt{m.zhukovskii@sheffield.ac.uk}, \orcidlink{0000-0001-8763-9533} }}

\date{}

\maketitle


%

%
%
%




\begin{abstract}
    We show that for any natural number~$s$, there is a~constant $\gamma$ and a~subgraph-closed class having, for any natural~$n$, at most $\gamma^n$ graphs on $n$ vertices up to isomorphism, but no adjacency labeling scheme with labels of size at most $s \log n$. In other words, for every~$s$, there is a~small --even \emph{tiny}-- monotone class without universal graphs of size $n^s$. Prior to this result, it was not excluded that every small class has an almost linear universal graph, or equivalently a~labeling scheme with labels of size $(1+o(1))\log n$. The existence of such a~labeling scheme, a~scaled-down version of the recently disproved Implicit Graph Conjecture, was repeatedly raised [Gavoille and Labourel, ESA '07; Dujmovi\'{c} et al., JACM~'21; Bonamy et al., SIDMA '22; Bonnet et al., Comb. Theory '22]. Furthermore, our small monotone classes have unbounded twin-width, thus simultaneously disprove the already-refuted Small conjecture; but this time with a~self-contained proof, not relying on elaborate group-theoretic constructions. 
 
 As our main ingredient, we show that with high probability an Erd\H{o}s--R\'{e}nyi random graph $G(n,p)$ with $p=O(1/n)$ has, for every $k \leqslant n$, at most $2^{O(k)}$ subgraphs on $k$ vertices, up to isomorphism.
 As a~barrier to our general method of producing even more complex tiny classes, we show that when $p=\omega(1/n)$, the latter no longer holds.
 More concretely, we provide an explicit lower bound on the number of unlabeled $k$-vertex induced subgraphs of $G(n,p)$ when $1/n \leq p \leq 1-1/n$.
 We thereby obtain a~threshold for the property of having exponentially many unlabeled induced subgraphs: if $\min \{p, 1-p\}<\delta/n$ with $\delta < 1$, then with high probability even the number of all unlabeled (not necessarily induced) subgraphs is $2^{o(n)}$, whereas if $C/n < p < 1-C/n$ for sufficiently large $C$, then with high probability the number of unlabeled induced subgraphs is~$2^{\Theta(n)}$. This result supplements the study of counting unlabeled induced subgraphs that was initiated by Erd\H{o}s and R\'{e}nyi with a~question on the number of unlabeled induced subgraphs of Ramsey graphs, eventually answered by Shelah.
	\end{abstract}

 \clearpage
	
	\section{Introduction}
	\label{sec:intro}

A \emph{class} of graphs is a~set of graphs which is closed under isomorphism. 
Let $\mathcal{C}$ be a~class of graphs and $f : \Nn \rightarrow \Nn$ be a~function. An \emph{$f(n)$-bit adjacency labeling scheme} or simply \emph{$f(n)$-bit labeling scheme}  for $\Cc$ is a~pair (encoder, decoder) of algorithms where for any $n$-vertex graph $G\in \Cc$ the encoder assigns to the vertices of $G$ \emph{pairwise distinct}\footnote{We note that this assumption is without loss of generality as the classes that we are interested in, namely hereditary classes with $2^{\Theta(n\log n)}$ labeled $n$-vertex graphs, are known to always contain arbitrarily large graphs whose all vertices have pairwise distinct neighborhoods \cite{Alekseev97} and therefore all of them should be assigned distinct labels in any adjacency labeling scheme. On the other hand, the assumption is convenient when considering the correspondence between labeling schemes and universal graphs.} $f(n)$-bit binary strings, called \emph{labels}, such that the adjacency between any pair of vertices can be inferred by the decoder only from their labels.
We note that the decoder depends on the class $\Cc$, but not on the graph $G$.
The function $f$ is the \emph{size} of the labeling scheme.
Adjacency labeling schemes were introduced by Kannan, Naor, and Rudich \cite{KNR88,KNR92},
and independently by Muller \cite{Muller88} in the late 1980s and have been actively studied since then.\footnote{It would be hard to give a~short but comprehensive list of references, so we refer the interested reader to the list of hundreds of papers citing \cite{KNR92}.}
%
For a~function $u : \Nn \rightarrow \Nn$, a~\emph{universal graph sequence} or simply \emph{universal graph} of size $u(n)$ is a~sequence
of graphs $(U_n)_{n\in \Nn}$ such that for every $n \in \Nn$ the graph $U_n$ has at most $u(n)$ vertices and every $n$-vertex graph in $\Cc$ is an induced subgraph of $U_n$.
It was observed in \cite{KNR92} that for a~class of graphs the existence an $f(n)$-bit labeling scheme is equivalent to the existence of a~universal graph of size $2^{f(n)}$.

One line of research in the area of labeling schemes is to describe graph classes that admit an
\emph{implicit representation}, i.e., an $O(\log n)$-bit labeling scheme.
Since labels assigned by a~labeling scheme are assumed to be pairwise distinct, $\lceil \log n \rceil$ is a~lower bound on the size of any labeling scheme (where, and throughout the paper, $\log$ is in base $2$).
Thus, understanding the expressive power of the labeling schemes of size of order $\log n$ is a~natural question.
A simple counting argument shows that a 
necessary condition for the existence of such a~labeling scheme for a~graph class
is at most \emph{factorial speed of the class}, i.e., the number of $n$-vertex graphs in the class should be at most $2^{O(n\log n)}$.
This condition alone is not sufficient, as witnessed by Muller's construction  \cite{Muller88} of a factorial graph class that does not admit an implicit representation. The crucial property used in that construction is that the class is not \emph{hereditary}, i.e., not closed under removing vertices.
Kannan, Naor, and Rudich \cite{KNR88} asked whether factorial speed and the property of being hereditary are
sufficient conditions for an implicit representation. 
This question was later reiterated by Spinrad \cite{Spinrad03} in the form of a~conjecture, that became known as the \emph{Implicit Graph Conjecture},

\begin{labeling}{(\textit{IGC}):}
	\item[(\textit{IGC}):] Any hereditary class of at most factorial speed admits an $O(\log n)$-bit labeling scheme.
\end{labeling}

The question remained open for three decades until the recent development of a~connection
between randomized communication complexity and adjacency labeling schemes \cite{Harms20,HWZ22,HHH22} led to the beautiful and striking refutation of the conjecture by Hatami and Hatami \cite{HH22}. In particular, Hatami and Hatami
showed that for any $\delta > 0$ there exists a~hereditary factorial class that does not admit
an $(n^{1/2-\delta})$-bit labeling scheme. This refutation leaves wide open the question of characterizing hereditary graph classes that admit $O(\log n)$-bit labeling schemes, and no plausible conjecture
is available at the moment. 

A large and important body of work within labeling schemes focuses on
the design of labeling schemes of asymptotically optimal size.
Such labeling schemes are known for various graph classes. For example, trees admit a~$(\log n + O(1))$-bit labeling scheme \cite{ADK17}, planar graphs admit a~$(1+o(1))\log n$-bit labeling scheme \cite{DEGJMM21}, all classes of bounded treewidth admit a~$(1+o(1))\log n$-bit labeling scheme \cite{GL07}. 
The latter example deals with an infinite family of graph classes rather than a~single class. The classes in this family are parameterized by the value of their treewidth, and the bound means that for any natural $k$ the class of graphs of treewidth at most $k$ admits a~labeling scheme of size $(1+o(1))\log n$, where $o(1)$
refers to a~function that depends on $n$ and $k$, and tends to 0 as $n$ goes to infinity.

We say that a~family $\Fc$ of graph classes has \emph{uniformly bounded} labeling schemes,
if there exists a~natural number $s$ such that every class in $\Fc$ admits an~
$(s + o(1))\log n$-bit labeling scheme. Using this terminology, we can say that the family of classes of graphs of bounded treewidth has uniformly bounded labeling schemes.
Another example of a~family with uniformly bounded labeling schemes is the family of proper minor-closed classes, since each of these classes admits a~$(2+o(1))\log n$-bit labeling scheme \cite{GL07}.

Note that an $n$-vertex graph with adjacency labels of size $f(n)$ can be encoded by a~binary
word of length at most $n \cdot f(n)$ obtained by concatenating the vertex labels.
Therefore, an $f(n)$-bit labeling scheme can represent at most $2^{n \cdot f(n)}$ unlabeled $n$-vertex graphs. Hence, a~necessary condition for a~family of graph classes to have a~uniform bound $(s+o(1))\log n$ on labeling schemes is that every class in the family should have at most $2^{(s+o(1))n \log n}$ unlabeled $n$-vertex graphs.
The proper minor-closed classes (including the classes of bounded treewidth) are \emph{small}, i.e.,
in each such class the number of \emph{labeled} $n$-vertex graphs is at most $n!c^{n} = 2^{(1+o(1))n \log n}$ for some constant $c$ \cite{NSTW06}. Thus, this family has the uniform bound of $2^{(1+o(1))n \log n}$ even on the number of labeled graphs.
In fact, the bound on the number of \emph{unlabeled} graphs in proper minor-closed classes is smaller: 
for any such
class there exists a~constant $c$ such that the number of unlabeled $n$-vertex graphs is at~most~$c^n$ \cite{BlankenshipThesis,BNW10,AFS12}.
We will call such classes \emph{tiny}. 
Note that any $n$-vertex unlabeled graph corresponds to at most $n!$ labeled graphs that are isomorphic to it, and therefore any tiny class is small. It is not known whether the converse holds (see $3. \Rightarrow 2.$ of Conjecture 8.1, in the first version of~\cite{BNO21}).

It is known that the classes of bounded clique-width \cite{ALR09} and more generally the classes of bounded twin-width \cite{BNO21} 
are tiny, and thus they have a~uniform upper bound on the number of $n$-vertex graphs.
These classes are also known to admit $O(\log n)$-bit labeling schemes, but these are not 
uniformly bounded labeling schemes.
Indeed, the best known labeling schemes for graphs of clique-width at most $k$ are of size $\Theta(k \log k \cdot \log n)$ \cite{Spinrad03,Banerjee22}; and the only known labeling scheme for graphs of twin-width at most $k$ is of size $2^{2^{\Theta(k)}} \cdot \log n$ \cite{BGK22}.
For each of these families no uniform bound is known, but such a bound is still possible. 
In fact, the necessary condition on the number of graphs does not rule out that the entire family of small
classes admits the optimal uniform bound of $(1+o(1))\log n$ on labeling schemes, and this question appeared in the literature several times \cite{GL07,DEGJMM21,Bonamy22,BGK22}. 

The main result of the present work is a~strong negative answer to this question: for any constant~$s$, there exists a~\emph{monotone tiny} class that does not admit an $s \log n$-bit labeling scheme.

	\subsection{Our contributions}
	Our main result shows that the family of small classes cannot have uniformly bounded adjacency labeling schemes.
In fact, we prove the stronger result that the \emph{monotone} (i.e., closed under taking subgraphs) tiny graph classes do not admit such labeling schemes.  

\def\mainthm{For any constant $s \in \Nn$, there exists a~tiny monotone graph class $\mathcal{C}$ that does not admit a~universal graph of size $n^s$.
	Equivalently, $\mathcal{C}$ has no adjacency labeling scheme of size $s \log n$.}

\begin{theorem}\label{thm:bdddeglabelings}
	\mainthm
\end{theorem}

Our technique also allows us to show that the family of all monotone tiny (and therefore the family of all small) classes cannot be ``described'' by a~countable set of $c$-factorial classes, where a~hereditary class $\Cc$ is $c$-factorial if it contains at most $n^{cn}$ labeled graphs on $n$ vertices.

\def\classthmbuff{	
	For any $c>0$, and any countable family $\mathcal{F}$ of $c$-factorial classes, there exists a~monotone tiny class that is not contained
	in any of the classes in $\mathcal{F}$.}
\begin{theorem}\label{th:class-complexity}
	\classthmbuff
\end{theorem} 
As noted in \cite{AAALZ23}, most positive results from the literature on
implicit representations are associated with specific graph parameters in
the sense that an implicit representation for a~class is derived from the fact
that some graph parameter is bounded for this class.
It is tempting to try and apply such a~parametric approach to the characterization of small classes: identify a~set of graph parameters each of which implies smallness, and such that for any small class, at least one of the parameters in the set is bounded. 
Such an attempt, in the somewhat extreme form of a~single graph parameter, was recently made.
Namely, the Small conjecture posited that every small hereditary class has bounded twin-width \cite{BGK22}.
This conjecture was already refuted in \cite{BGTT22} using a~group-theoretic construction.
A corollary (\cref{th:param-complexity}) of Theorem \ref{th:class-complexity} is that the parametric approach to characterization of small, and even monotone tiny classes, is futile even if countably many parameters are utilized. As a~special case of this corollary we obtain an alternative refutation of the Small conjecture that avoids groups.

\def\twwthm{There exists a~monotone tiny graph class with unbounded twin-width.}

\begin{corollary}\label{cor:small-false}
	\twwthm
\end{corollary}

In order to prove Theorem \ref{thm:bdddeglabelings} we construct the target classes via the probabilistic method. This required us to study of the growth of the number of $k$-vertex subgraphs in a random graph (see Section \ref{sec:outline} for more details).
In particular, we obtain two characterizations of the exponential growth of the number of subgraphs in random graphs. The first one, Theorem \ref{thm:tinyamalg},
characterizes random graphs in which the number of $k$-vertex subgraphs grows exponentially as a function of $k$. The second one, Theorem \ref{thm:threshold}, characterizes random graphs in which the \emph{total} number of (induced) subgraphs grows exponentially. 
We introduce some essential definitions before stating these results.

Let $G(n,p)$ be the (random graph) distribution on $n$-vertex graphs where each edge is included independently with probability $p$. For a~constant $c$, we say that an $n$-vertex graph~$G$ is \emph{monotone $c$-tiny} if, for every $k \geq 1$, the number of unlabeled $k$-vertex subgraphs of~$G$ is at~most~$c^k$. 

\begin{theorem}[Consequence of~Theorems~\ref{thm:Gnmisgood} and  \ref{thm:lowerbound}]\label{thm:tinyamalg}	
	The following holds:
	$\;$	
	\begin{enumerate}
		\item for every $d_1 \geq 0$, there exists $c > 0$ such that w.h.p. $G_n \sim G(n,d_1/n)$ is monotone $c$-tiny; and
		\item for every $c>0$, there exists $d_2\geq0$ such that w.h.p. $G_n \sim G(n,d_2/n)$ is not monotone $c$-tiny.
	\end{enumerate}

\end{theorem} 
 
%

To obtain the second half of this result (Theorem \ref{thm:lowerbound}), we provide an explicit lower bound on the number of unlabeled $k$-vertex induced subgraphs of $G(n,p)$ for all $1/n \leq p \leq 1-1/n$. 
As a~consequence, we prove that $p=\Theta(1/n)$ is a~threshold for the property of having exponentially many unlabeled (induced) subgraphs. This has connections to an old question by Erd\H{o}s and R\'{e}nyi,
who conjectured that Ramsey graphs (i.e., graphs in which maximum homogeneous induced subgraphs are of logarithmic size) have exponentially many unlabeled induced subgraphs. This conjecture was proved by Shelah~\cite{S1998}.	
Our threshold shows that (relatively) sparse $G(n,p)$ has exponentially many unlabeled subgraphs (note that such graphs are not Ramsey). 
Let $s(G_n)$ and $i(G_n)$ denote the number of unlabeled subgraphs and the number of unlabeled \emph{induced} subgraphs of $G_n$, respectively.
\def\threshold{Let $\delta>0$ be any constant, $G_n \sim G(n,p)$, and let $C>0$ be sufficiently large. 
	\begin{enumerate}[$(i)$]
		\item\label{itm:1} If $\min \{p, 1-p\}<\frac{1-\delta}{n}$, then w.h.p., $i(G_n)\leq s(G_n)=2^{o(n)}$.
		\item\label{itm:2} If $\frac{C}{n} < p < 1-\frac{C}{n}$, then there exists $\lambda>1$ such that w.h.p., $s(G_n)\geq i(G_n)>\lambda^n$.
	\end{enumerate}
} 
\begin{theorem}\label{thm:threshold}
	\threshold
\end{theorem}


\subsection{Proof outlines and techniques}\label{sec:outline}

\begin{figure}
	\centering
	\includegraphics[width=1\textwidth]{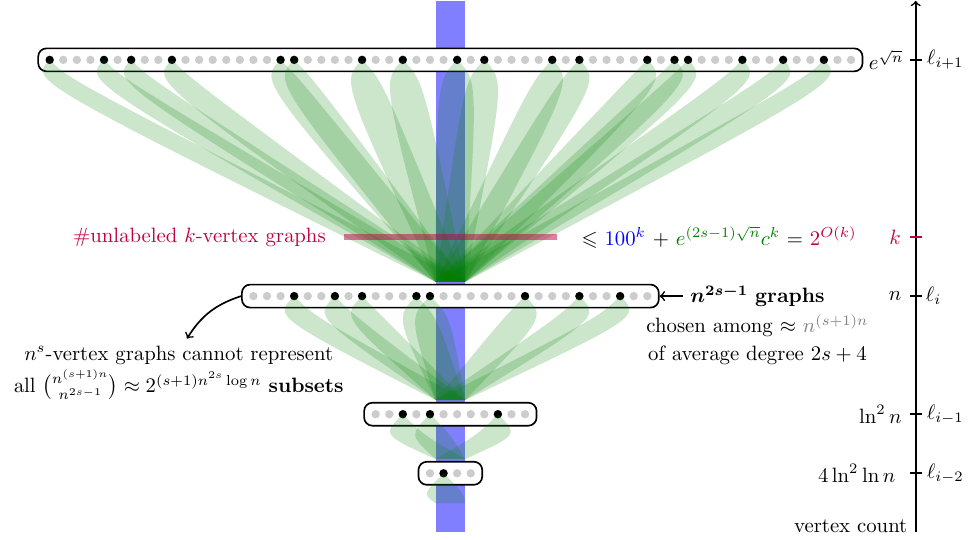}
	\caption{
		How to build a~monotone tiny graph class without $n^s$-vertex universal graphs?
		Take the subgraph closure of a~typical subset of $n^{2s-1}$ graphs with $n$ vertices and average degree $2s+4$, for every element $n$ of a~sequence $(\ell_i)_{i \in \mathbb N}$ with $\ell_0=1$ and $\ell_{i+1} = \lceil e^{\sqrt{\ell_i}} \rceil$ for every $i \in \mathbb N$.
		The tinyness is explained by elements in color (\textcolor{purple}{purple}, \textcolor{blue}{blue}, \textcolor{green!50!black}{green}).
		Indeed, the subgraphs on at most $t(n) = \ln^2 n$ vertices of such graphs turn out to have all their $k$-vertex connected components with less than $k+k/\ln k$ edges; \textcolor{blue}{the class $\mathcal Y$ of graphs with such a~limited edge--vertex surplus} (blue strip) is shown to have growth less than $k \mapsto 100^k$, thus is tiny.
		And, in general, the set of \textcolor{green!50!black}{all the subgraphs} (each green ``cone'') of \emph{one} such graph contains, for every $k$, at most $c^k$ graphs (with $c=s^{O(s)}$).
		The \textcolor{purple}{overall tinyness} is obtained since $k$-vertex graphs with $k \in [n,e^{\sqrt n}]$, outside the blue strip, stem from $(e^{\sqrt n})^{2s-1}=2^{o(n)}$, hence $2^{o(k)}$ graphs.  
		The non-representability is explained in \textbf{bold} and \textcolor{gray}{gray}.
		Of the $\approx n^{(s+1)n}$ graphs with $n$~vertices and average degree $2s+4$ (\textcolor{gray}{points} within the rounded rectangular horizontal boxes), we have around ${n^{(s+1)n} \choose n^{2s-1}} \approx 2^{(s+1)n^{2s}\log n}$ \textbf{choices} of subsets of cardinality $n^{2s-1}$. \newline
		This is more than \emph{all} the $n^s$-vertex universal graphs can represent since $2^{n^{2s}} \cdot {{n^s \choose n} \choose n^{2s-1}} \approx 2^{sn^{2s}\log n}$.
	}
	\label{fig:overall-picture}
\end{figure}

In this section we outline the proofs of our main results.

\paragraph{Monotone tiny classes that do not admit some polynomial size universal graphs.}
Our approach is inspired by the refutation of the IGC by Hatami and Hatami \cite{HH22}.
Namely, we expose a~family of \emph{tiny monotone} classes that is so large that there are not enough universal graphs of uniformly bounded size to capture all of them.
The approach is illustrated in \cref{fig:overall-picture} and involves several key ingredients:
\begin{enumerate}
	\item Estimation of the number of sets of graphs of fixed cardinality representable by universal graphs. A~set of graphs $\Mc$ is \emph{representable} by a~universal graph $U$, if every graph in $\Mc$ is an induced subgraph of $U$.
	Fix some positive integer $s$.
	A~direct estimation shows that the number of sets of cardinality $k_n := n^{2s-1}$ of $n$-vertex graphs that are representable by a~$u_n$-vertex universal graph, with $u_n := n^s$ is at~most 
	\begin{equation}\label{eq:num-representable}
		2^{u_n^2} \cdot u_n^{nk_n} = 2^{n^{2s}} \cdot n^{s n^{2s}} = 2^{n^{2s} + s n^{2s} \log n}.
	\end{equation}
	
	\item Notion of \emph{$c$-tiny graphs} and its refinement \emph{$(c,\Yc,t)$-tiny graphs}. 
	We will construct our monotone tiny classes by taking the monotone closure of an appropriately chosen set of graphs.
	The monotonicity and tinyness of target classes impose a~natural restriction on graphs that can be used in such constructions.
	To explain, let $\Cc$ be a~monotone and tiny class, i.e., for every graph
	$G \in \Cc$ all subgraphs of $G$ are also in $\Cc$, and  for every natural $n$, the class $\Cc$ contains at most~$c^n$ unlabeled $n$-vertex graphs for some constant $c$.
	This implies that for any $G \in \Cc$ and every $k$, the number of unlabeled $k$-vertex subgraphs of $G$ is at most $c^k$.
	We call graphs possessing this property \emph{monotone $c$-tiny}.
	
	This notion, however, is not strong enough for our purposes.
	Indeed, while each
	monotone $c$-tiny graph contributes to the monotone closure an appropriate number of graphs at every \emph{level} (i.e., on every number of vertices), we build our desired classes by taking the subgraph-closure of \emph{infinitely} many of such graphs,
	and this can result in some levels having too many graphs.
	To overcome this difficulty, we introduce the notion of \emph{monotone $(c,\Yc,t)$-tiny} graphs, which are monotone $c$-tiny $n$-vertex graphs with the extra restriction that all their subgraphs on at most $t(n)$ vertices belong to a~\emph{fixed} tiny class~$\Yc$.
	
	\item Construction of monotone tiny classes from sets of monotone $(c,\Yc,t)$-tiny graphs.
	For some suitable $\Yc$ and $t$, we show that for any constant $c$ and sequence $(\Mc_{_{\ell_i}})_{i \in \Nn}$, where $\ell_i\leq t(\ell_{i+1})$ and each $\Mc_{\ell_i}$ is a~set of monotone $(c,\Yc,t)$-tiny $\ell_i$-vertex graphs of cardinality~$k_{\ell_i}$ (recall that $k_n = n^{2s-1}$), the monotone closure $\mon(\cup_{i \in \Nn} \Mc_{\ell_i})$ is a~tiny class.
	
	\item Lower bound on the number of sets of cardinality $k_n$ of monotone $(c,\Yc,t)$-tiny $n$-vertex graphs. 
	We show that for any constant $d$ there exists a~constant $c$ such that the number of unlabeled monotone $(c,\Yc,t)$-tiny $n$-vertex graphs grows as $n^{\frac{(d-2)n}{2} (1-o(1))}$, and thus the number of sets of cardinality $k_n$ of $(c,\Yc,t)$-tiny $n$-vertex graphs is at least 
	\begin{equation}\label{eq:num-good-sets}
		2^{\Omega\left( k_n \frac{(d-2)n}{2} \log n \right)} = 2^{\Omega\left(\frac{d-2}{2} n^{2s} \log n \right)}.
	\end{equation}
	\label{itme:many-sets}
\end{enumerate}

By choosing $d$ larger than $2s+2$, it can be seen that \cref{eq:num-good-sets} is larger than \cref{eq:num-representable}. 
Therefore, there exists a~monotone tiny class $\mon(\cup_{i \in \Nn} \Mc_{\ell_i})$ that is not representable by a~universal graph of size~$n^s$.

\paragraph{Many $(c,\Yc,t)$-tiny graphs.}
A core step in the above approach is to show that for any constant~$d$ there exists a~constant
$c$ such that the number $(c,\Yc,t)$-tiny graphs grows as $2^{\frac{d-2}{2}(1-o(1))n\log n}$.
To do so, we show that a~random graph $G \sim G(n, d/n)$ is $(c,\Yc,t)$-tiny w.h.p.. 
For this we set $t(n) := \ln^2 n$ and choose $\Yc$ to be a~monotone class of graphs in which the number of edges does not exceed the number of vertices by a~large margin.
First, we show that $\Yc$ is tiny.
Next, we show that w.h.p. every subgraph of $G$ on at most $t(n)$ vertices belongs to $\Yc$.
Finally, we show that for any $k \geq t(n)$ there are at most exponentially many unlabeled subgraphs of $G$ on $k$ vertices. 
This is achieved by computing the expected number of such subgraphs which are connected and have a~fixed number of edges, and then by showing that these numbers are sufficiently concentrated.

\paragraph{Barrier for $c$-tinyness of random graphs.}
Our lower bound on the size of labeling schemes is dependent on the number of $(c,\Yc,t)$-tiny graphs,
in the sense that the more $(c,\Yc,t)$-tiny $n$-vertex graphs we have, the better our lower bound is. In particular, it is not excluded that using our approach a~lower bound can be made sufficiently strong to disprove the IGC with a~tiny class.
One would then need to show that there are superfactorially many, i.e., $n^{\omega(n)}$, $(c,\Yc,t)$-tiny $n$-vertex graphs. 
We provide evidence suggesting that this cannot be achieved via random graphs.
More specifically, we show that for any fixed $c$ and any $p = \omega(1/n)$
w.h.p. a~random graph drawn from $G(n, p)$ is \emph{not} $c$-tiny. 

Our strategy is as follows. 
First, we observe that for a~large enough $k=o(n)$, w.h.p. a~typical $k$-vertex induced subgraph of $G(n,p)$ has at least linearly in $n$ many edges.
Next, assuming that the number of unlabeled $k$-vertex induced subgraphs (i.e., the number of isomorphism classes of $k$-vertex induced subgraphs) is small, we conclude that there is a~\emph{large} isomorphism class, i.e., an isomorphism class that is represented by many $k$-vertex induced subgraphs.
Then, we show that w.h.p. two induced $k$-vertex subgraphs with a~``small'' vertex intersection cannot be isomorphic due to the linear in $n$ lower bound on the number of edges in each. 
Hence, the $k$-subsets of vertices inducing the representatives of the large isomorphism class
should each have a~``large'' intersection with a~fixed $k$-subset.
To reach a~contradiction, we show that the total number of $k$-subsets that have ``large''
intersections with a~fixed $k$-subset is much smaller than the number of the representatives of
the large isomorphism class.


Our proof provides an explicit lower bound on the number of $k$-vertex induced subgraphs in $G(n,p)$.
From this bound we derive a~coarse $\Theta(1/n)$-threshold for $G(n,p)$ to have exponentially many unlabeled induced subgraphs. Furthermore, it turns out that the property of having exponentially many unlabeled (not necessarily induced) subgraphs has the same (coarse) threshold.

\medskip

\textbf{Organization.}
The rest of the paper is organized as follows.
\cref{sec:notation} contains classic definitions and inequalities, as well as the crucial definitions of (monotone) $c$-tiny and (monotone) $(c,\Yc,t)$-tiny graphs.
\cref{sec:Gnp-subgraphs} shows upper and lower bounds on the number of unlabeled subgraphs of Erd\H{o}s--Rényi random graphs.
Note that only the upper bound is used in the proof of our main result, Theorem~\ref{thm:bdddeglabelings}.
In \cref{sec:monotone-tiny} we give the construction of our complex monotone tiny classes.
Finally we discuss some open problems in~\cref{sec:discussion}.

To establish the proof of Theorem~\ref{thm:bdddeglabelings}, one only needs to go through~\cref{sec:notation,sec:upper-bound,sec:general-framework,subsec:lbls}.
	\section{Preliminaries, inequalities, and tinyness}
	\label{sec:notation}
		
	We begin with some routine definitions in \Cref{sec:basics}, then proceed with a~number of standard but useful inequalities in \Cref{sec:ineq}. Finally, \Cref{sec:tiny} introduces our new notions of tinyness for graphs and classes; these definitions are key to how we build our classes and prove our main result.  
	
	\subsection{Standard definitions and notation }\label{sec:basics}
	For two real numbers $i,j$, we let $[i,j]:=\{\lceil i\rceil,\lceil i\rceil+1,\ldots,\lfloor j \rfloor-1,\lfloor j\rfloor \}$.
	Note that if $j < i$, then $[i,j]$ is the empty set.
	We may use $[i]$ as a~shorthand for $[1,\lfloor i\rfloor ]$, and $\ln^2 x$ as a~shorthand for $(\ln x)^2$. We use the notation $X\sim \mathcal{D}$ to denote that the random variable $X$ has distribution $\mathcal{D}$. We say that a sequence of events $(A_n)$ holds \emph{with high probability (w.h.p.)} if $\Pr{A_n}\rightarrow 1$ as $n\rightarrow \infty$.

	\paragraph{Graphs.} We consider finite undirected graphs, without loops or multiple edges.
	Given a~graph~$G$, we write~$V(G)$ for its vertex set, and~$E(G)$ for~its edge set.	
	A graph $H$ is a~\emph{subgraph} of~$G$ if $V(H) \subseteq V(G)$ and $E(H) \subseteq E(G)$.
	Thus, $H$ can be obtained from $G$ by vertex and edge deletions.
	The graph $H$ is an \emph{induced subgraph} of $G$ if $V(H) \subseteq V(G)$,
	and~$E(H)$ consists exactly of the edges in~$E(G)$ with both endpoints in~$V(H)$.
	In that case, $H$ can be obtained from $G$ by vertex deletions only. In the usual way, for a~set of vertices $U\subseteq V(G)$, we denote by $G[U]$ the induced subgraph of $G$ with the set of vertices $U$.
	We denote by $s_k(G)$ and $i_k(G)$ the numbers of unlabeled $k$-vertex subgraphs and of unlabeled $k$-vertex induced subgraphs of $G$, respectively.
	For $k\geq 1 $ and $t\geq 0$, we denote by $s_k(G,t)$ the number of unlabeled $k$-vertex $t$-edge subgraphs of $G$. 
	
	When we refer to an $n$-vertex graph $G$ as \emph{labeled}, we mean that the vertex set of $G$
	is $[n]$, and we distinguish two different labeled graphs even if they are isomorphic.
	In contrast, if we refer to $G$ as \emph{unlabeled} graph, its vertices are indistinguishable and
	two isomorphic graphs correspond to the same unlabeled graph. For this reason, we will sometimes refer to unlabeled graphs as \emph{isomorphism classes}.

	\paragraph{Graph Classes.} A~class of graphs is \emph{hereditary} if it is closed under taking induced subgraphs, and it is \emph{monotone} if it closed under taking subgraphs.
	For a~set $\mathcal{C}$ of graphs we let $\her(\mathcal{C})$ denote the hereditary closure of $\mathcal{C}$, i.e., the inclusion-wise minimal hereditary class that contains $\mathcal{C}$; and $\mon(\mathcal{C})$ denote the monotone closure of $\mathcal C$, i.e., the minimal monotone class that contains $\mathcal{C}$. 
	A graph class is \emph{small} if it is hereditary and there exists a~constant $c$ such that the number
	of $n$-vertex \emph{labeled} graphs in the class is at most $n!c^n$ for every $n \in \Nn$.
	
	
	\subsection{Inequalities} \label{sec:ineq} 
	We will make use of the inequalities 
	$\left( \frac{n}{k}\right)^k\leq \binom{n}{k} \leq \left( \frac{ne}{k}\right)^k$; see for example \cite[Lemma 2.7]{FriezeKaronski}. One such example is to prove the following bound on the number of sparse graphs. 
	
	\begin{lemma}\label{lem:sparseissmall}
		For any $k\geq 1$, the number of unlabeled \emph{connected} $k$-vertex graphs with at most $k-1+ \frac{k}{\ln k}$ edges is less than $100^k$. 
	\end{lemma}
	
	\begin{proof}
		A~connected $k$-vertex graph with at most $k-1+ \frac{k}{\ln k}$ edges consists of a~spanning tree (of $k-1$ edges) plus at most $\frac{k}{\ln k}$ edges.
		Recall that the number of unlabeled trees on $k \geq 1$ vertices is at most~$k \cdot 4^k$~\cite{TreesOtter}.
		The number of ways to choose the remaining at most $\frac{k}{\ln k}$ edges is bounded from above by
		\[\sum_{i=0}^{\frac{k}{\ln k}}\binom{\binom{k}{2}}{i} \leq \frac{k}{\ln k}\cdot \binom{\binom{k}{2}}{\frac{k}{\ln k}}\leq k\cdot \left( \frac{k^2e   }{2\frac{k}{\ln k}} \right)^{\frac{k}{\ln k}} < k\cdot k^{2 \frac{k}{\ln k}} = k \cdot e^{2k}.\]
		The last inequality holds since $e \ln k < 2k$ for every $k \geq 1$. The result follows since \[(k \cdot 4^k) \cdot (k \cdot e^{2k})\leq k^2 \cdot 30^k \leq 3^k \cdot 30^k \leq 100^k,\] for every $k \geq 1$ as $4e^2\leq 29.6$.
	\end{proof}
	
	We use the following form of Chernoff bound, which is an immediate consequence of ~\cite[Theorem 2.1]{Janson}.
	
	\begin{lemma}[Chernoff bound]\label{lem:Chb}
		Let $\xi$ be a~binomial random variable with $\mathbb{E}\xi=\mu$ and let $t\geq 0$. Then
		$$
		\Pr{|\xi - \mu| \geq t}\leq2\exp\left[-\frac{t^2}{2(\mu+t/3)}\right].
		$$
	\end{lemma}
	The following reduces the estimation of the number of graphs in a~class to that of its number of \emph{connected} graphs, whenever it is convenient.

	\begin{lemma}\label{lem:connected}
		Let $\Cc$ be a~hereditary class of graphs such that, for every $n$, the number of unlabeled \emph{connected} $n$-vertex graphs in $\Cc$ is at most $\alpha^n$ for some constant $\alpha$. Then, for every $n$, the number of \emph{all} unlabeled $n$-vertex graphs in $\Cc$ is at most $\beta^n$, for some $\beta:=\beta(\alpha)$.
	\end{lemma}
	\begin{proof}For any $k_1, k_2, \ldots, k_r > 0$ such that $k_1+ k_2+ \cdots + k_r =k$, the number of unlabeled \mbox{$k$-vertex} graphs in $\Cc$ with $r$ connected components of sizes $k_1, k_2, \ldots, k_r$, respectively, is at most $\prod_{i=1}^r\alpha^{k_i} = \alpha^k$.
		Furthermore, since the \emph{partition function} $p(k)$ (counting the number of partitions of a~given positive integer $k$ as a~sum of positive integers) satisfies $p(n) <  \exp(\pi\cdot  \sqrt{2k/3})$~\cite[Eq.~(4)]{Erdos}. We conclude that the total number of \mbox{$k$-vertex} graphs in $\Cc$ is bounded above by $\beta^k$ for some $\beta:=\beta(\alpha)$.
	\end{proof}

	\subsection{Definitions of tinyness for graphs and classes}\label{sec:tiny}

	\begin{definition}[Tiny class]\label{def:tiny}
		We say that a~class $\mathcal C$ of graphs is \emph{tiny} if it is hereditary, and there is a~constant $c$ such that
		the number of unlabeled $n$-vertex graphs in $\mathcal C$ is at most $c^n$ for every $n \in \Nn$.
	\end{definition}
	
	Our next definition can be thought of as a~specialization of the definition of tiny classes to graphs.
	In particular, both the hereditary closure of a~\emph{$c$-tiny graph} and the monotone closure of a~\emph{monotone $c$-tiny graph} give a~tiny class.

	\begin{definition}[$c$-tiny graphs]\label{def:c-small}
		Let $c$ be a~constant.
		An $n$-vertex graph $G$ is \emph{$c$-tiny}  if $i_k(G) \leq c^k$ holds for every $k \in [n]$.
		The graph $G$ is \emph{monotone $c$-tiny} if $s_k(G) \leq c^k$ holds for every $k \in [n]$.
	\end{definition}
	
	The next definition is a~strengthening of the previous one.
	
	\begin{definition}[$(c,\Yc,t)$-tiny graphs]\label{def:good}
		Let $c$ be a~constant, $\Yc$ be a~tiny class of graphs, and
		let $t : \Nn \rightarrow \Nn$ be a~non-decreasing and unbounded function. We say that an $n$-vertex graph $G$ is \emph{$(c,\Yc,t)$-tiny} (respectively, \emph{monotone $(c,\Yc,t)$-tiny}) if 
		\begin{compactenum}[(1)]
			\item\label{itm:good1} for every $k \in [t(n)]$, every $k$-vertex induced subgraph (respectively, subgraph) of $G$ is in $\Yc$; and
			\item\label{itm:good2} for any integer $t(n) <k \leq n $, it holds that $i_k(G) \leq c^k$ (respectively, $s_k(G) \leq c^{k}$).
		\end{compactenum}
	\end{definition}
	
	\begin{remark}\label{rem:connectedsuffices}
		We note that, since the sets of induced subgraphs and subgraphs of $G$ form hereditary classes,  in item (\ref{itm:good2}) of~\cref{def:good} the numbers of unlabeled induced subgraphs and subgraphs (i.e., $i_k(G)$ and $s_k(G)$, respectively) can be replaced with the numbers of unlabeled \emph{connected} induced subgraphs and subgraphs.
		This is because, by~\cref{lem:connected}, an exponential upper bound on the number of unlabeled \emph{connected} graphs implies an exponential bound on the number of arbitrary unlabeled graphs in any hereditary class. 
	\end{remark}
	
	The strengthening of the notion of $c$-tiny graphs to $(c,\Yc,t)$-tiny graphs will grant us a~lot of power when taking unions of many graphs of different sizes. In particular, when considering subgraphs of a~given size in the monotone closure of a~sequence of $(c,\Yc,t)$-tiny graphs, one does not need to worry about subgraphs coming from graphs significantly further along in the sequence (of much larger size) as all such subgraphs will belong to the fixed class $\mathcal{Y}$.
	In our application of $(c,\Yc,t)$-tiny graphs, the role of the fixed class $\Yc$ will be played by the following
	class $\Sc_d$ (where `s' stands for \emph{sparse}).
	\begin{definition}[Class $\Sc_d$]\label{def:classY}
		For any $d\geq 0$, let $\Sc_d$ be the class of all unlabeled graphs $G$ such that for every $k \in \left[1000^{10(d+1)}, |V(G)| \right]$, 
		every $k$-vertex subgraph of $G$ has at most $k - 1 + k/\ln k$ edges.
	\end{definition}
	
	We now show this class is suitable for use in the definition of monotone $(c,\Yc,t)$-tiny. 
	
	\begin{lemma}\label{lem:YmonTiny}
		For any fixed $d\geq 0$, the class $\Sc_d$ is monotone tiny.
	\end{lemma}
	\begin{proof}
		The fact that the class $\Sc_d$ is monotone easily follows from its definition.
		
		By \Cref{lem:connected}, it suffices to upperbound the number of connected graphs in $\mathcal{S}_d$. Since each connected graph in $\Sc_d$ on $k$ vertices either has at most $k-1+k/\ln k$ edges, or $k\leq 1000^{10(d+1)}$, the number of connected unlabeled $k$-vertex graphs in $\mathcal{S}_d$ is bounded from above by \[ 100^k +2^{\binom{1000^{10(d+1)}}{2}}\leq \eta^k ,\] for some constant $\eta=\eta(d)$, due to~\cref{lem:sparseissmall}. 
	\end{proof}

	\section{Growth of the number of unlabeled subgraphs in $G(n,p)$}
	\label{sec:Gnp-subgraphs}
 In \Cref{sec:upper-bound} we show that w.h.p.\ a~sparse random graph is  monotone $(c,\Sc_d,\ln^2)$-tiny. This result is used to construct a~large family of monotone tiny classes in the proof of Theorem~\ref{thm:bdddeglabelings}. Then, in \Cref{sec:lower-bound}, we prove a~lower bound on the constant $c:=c(d)$ for which a~random graph with average degree $d$ can be $c$-tiny w.h.p. This result shows that (at least w.h.p.) $c$ must grow with $d$, and thus gives strong evidence that our construction cannot be so easily improved by taking random graphs of growing average degree. Finally, in \Cref{sec:threshold}, we show that $p= \Theta(1/n)$ is a~threshold for having an exponential number of unlabeled (induced) subgraphs in a~random graph $G(n,p)$ and relate this to previous results on the number of unlabeled subgraphs of a~given graph. 
\subsection{Upper bounds}
\label{sec:upper-bound}
Let $G(n,m)$ denote the uniform distribution on simple graphs with $n$~vertices and $m$~edges.	Our aim in this subsection is to prove the following result.

\begin{theorem} \label{thm:Gnmisgood} 
	For any fixed $d\geq 0$, there exists a~constant $c:=c(d)$, such that for $G_n\sim G\!\left(n,\left\lceil d(n-1)/2 \right\rceil \right)$ we have 
	\[\Pr{G_n\text{ is \emph{not} monotone $(c,\Sc_d,\ln^2)$-tiny} } < 200 \cdot \sqrt{\frac{d}{n}}.\]   		
\end{theorem}

Recall that $G(n,p)$ denotes the distribution on $n$-vertex graph where each edge is included independently with probability $p$, i.e., a~fixed labeled graph $G$ on $[n]$ occurs with probability $p^{|E(G)|}(1-p)^{\binom{n}{2} -|E(G)|}$. We  work with $G(n,p)$ for most of this section as it is more convenient; this is not a~significant restriction as we can transfer the result to $G(n,m)$ via the following result.

\begin{lemma}\label{lem:transfer}
	Let $\mathcal{P}$ be any graph property (i.e., graph class) and $0\leq p\leq 1$ satisfy $p\binom{n}{2}\rightarrow \infty$ and $\binom{n}{2} -p\binom{n}{2}\rightarrow \infty$ and $m =\left\lceil p\binom{n}{2} \right\rceil $. Then, for $G_n \sim G(n,m)$ and $G_n' \sim G(n,p)$, we have  
	\[ \Pr{G_n\in \mathcal{P}} \leq 10 \sqrt{m}\cdot  \Pr{G_n' \in \mathcal{P}}. \]	
\end{lemma}

\noindent
Lemma \ref{lem:transfer} follows by a~very minor adaption of \cite[Lemma 3.2]{FriezeKaronski}; the only difference is a~ceiling in the number of edges, which makes no difference in the proof.

Recall that for $k\geq 1 $ and $t\geq 0$,  $s_k(G,t)$ denotes the number of unlabeled subgraphs $H$ of a~graph $G$ such that $|V(H)|=k$ and $|E(H)|=t$. \Cref{lem:smallkGnp} deals with small subgraphs.  
\begin{lemma}\label{lem:smallkGnp}
	Let $d\geq 0$ be fixed, $n\geq 3$, and $ t\geq 2$  satisfy $3\ln ((d+1)e^2) \ln t \leq \ln n$.
	Then, for $G_n\sim G\left(n,\tfrac{d}{n}\right)$, we have \[\Pr{\sum_{k=2}^t\;\sum_{j\geq k+\frac{k}{\ln k}} s_k\!\left(G_n,j\right) >0}\leq  \frac{15}{n}. \]  
\end{lemma}

\begin{proof}
	Observe there are $\binom{\binom{k}{2}}{k+t}$ ways of finding $k+t$ edges on a~$k$-vertex subset. Thus, the probability $G_n\sim G\left(n,\tfrac{d}{n}\right)$ has an $k$-vertex subgraph with at least $k+t$ edges is bounded from above by 
	\begin{align}\label{eq:p(s,t)}
		\Pr{\sum_{j\geq k+ t} s_k\!\left(G_n,j\right) >0} &\leq
		\binom{n}{k} \binom{\binom{k}{2}}{k+t}   \left(\frac{d}{n}  \right)^{k+ t}\notag \\    
		&\leq \left(\frac{ne }{k}\right)^k   \left(\frac{k^2 e}{2(k+t)}  \right)^{k+t} \left(\frac{d}{n}  \right)^{k+t}\notag \\  
		&\leq  \left(de^2 \right)^{k+t} \left(\frac{ k }{n }\right)^t\notag \\    
		&= \exp\left(k\ln \left( de^2 \right) - t\left( \ln n- \ln \left(kde^2\right) \right) \right).       \end{align} 
	
	Now, for any $k \leq t$ and $d\geq 0$, the assumptions on $t,d$ and $n$ give     
	\begin{equation*}\label{eq:logi}
		\ln (kde^2)  \leq \ln ((d+1)e^2)+\ln t \leq \ln ((d+1)e^2)\ln t \leq \frac{\ln n}{3}.
	\end{equation*}
	Thus, by \eqref{eq:p(s,t)}, if we set $t = \frac{k}{\ln k}$, then for any $k \leq t$,
	\begin{align}
		\Pr{\sum_{j\geq k+ \frac{k}{\ln k}} s_k\!\left(G_n,j\right) >0}&\leq  \exp\left(k\ln \left( de^2 \right) - \frac{k}{\ln  k }\cdot \left( \ln n- \ln \left(kde^2\right) \right) \right) \notag\\
		&\leq  \exp\left(k\ln \left( de^2 \right) - \frac{k}{\ln  k }\cdot \frac{2\ln n}{3} \right)\notag\\
		&=  \exp\left(\frac{k}{\ln k }\left(\ln \left( de^2 \right)\ln k -   \frac{2\ln n}{3} \right)\right)\notag\\
		& \leq \exp\left(- \frac{k}{\ln k }\cdot \frac{\ln n}{3} \right) = n^{- \frac{k}{3\ln k }}.\label{eq:pstbound}
	\end{align} 
	For $ k\in \{2,3,4\}$, all $k$-vertex graphs have less than $k+k/\ln k$ edges. Thus, by \eqref{eq:pstbound} and the union bound,   
	\begin{align*}
		\Pr{\sum_{k=2}^t\; \sum_{j\geq k+ \frac{k}{\ln k}} s_k\!\left(G_n,j\right) >0} &\leq 0+
		\sum_{k=5}^t\; \Pr{\sum_{j\geq k+ \frac{k}{\ln k}}  s_k\!\left(G_n,j\right) >0}  \leq \sum_{k=5}^t n^{- \frac{k}{3\ln k }} \notag \\
		&\leq   \frac{1}{n}\cdot \sum_{k=5}^\infty n^{- \frac{k}{3\ln k }+1}\notag\\
		&\leq \frac{15}{n},\notag \end{align*}where the last line follows since $n\geq 3$. 
\end{proof}

The following lemma covers any subgraphs of sufficiently large density. 

\begin{lemma}\label{lem:largekGnp}For any $d\geq 0$, there exists $\beta:=\beta(d)$ such that for any $n,k\geq 1$, $\delta\geq 1+1/\ln k$ such that $\delta k$ is an integer, and $G_n\sim G\left(n,\tfrac{d}{n}\right)$ we have \[\Pr{s_k\!\left(G_n,\delta k\right) \geq \beta^k} \leq  2^{-k}. \] 
\end{lemma} 

\begin{proof} We begin by bounding the expected number of  subgraphs appearing in $G_n\sim G\left(n,\tfrac{d}{n}\right)$. The result will then follow from Markov's inequality. Observe that
	\begin{align*}\Ex{s_k\!\left(G_n,\delta k\right)} &\leq \binom{n}{k} \binom{\binom{k}{2}}{\delta k}  \left(\frac{d}{n}\right)^{\delta k}\\
		&\leq \left(\frac{ne}{k}\right)^k  \left(\frac{k^2e}{2\delta k}\right)^{\delta k}\left(\frac{d}{n}\right)^{\delta k}\\ 
		&\leq  \left(\frac{k}{n}\right)^{(\delta -1)k} \cdot \left(\frac{e^2d}{\delta}\right)^{\delta k } . 
	\end{align*} 
	Next, we show that, for any fixed $d\geq 0$, there exists some constant $\alpha:=\alpha(d) \geq 1$ such that for all $k \in [n]$ and $1+\frac{1}{\ln k} \leqslant \delta \leqslant \frac{k-1}{2}$, the following inequality holds  
	\begin{equation}\label{eq:master}
		\left(\frac{k}{n}\right)^{(\delta -1)k}  \cdot \left(\frac{e^2d}{\delta}\right)^{\delta k }\leq \alpha^k . 
	\end{equation}
	If $\delta \geq e^2 d$ then, since $k\leq n$, the left-hand side of \eqref{eq:master} is at most $1^{(\delta-1)k}\cdot 1^{\delta k} \leq 1 $ and so we can assume that $\delta < e^2 d$ when proving \eqref{eq:master} holds. Note that we can also assume $d\neq 0$ or else the statement of the theorem holds vacuously.

	Since $\delta \geq 1 + 1/\ln k >1$, taking the $(\delta-1)k$-th root of both sides and rearranging gives 
	\begin{equation}\label{eq:master2}k \leq n \cdot \left( \alpha \cdot \left(\frac{\delta }{e^2 d} \right)^{\delta}  \right)^{\frac{1}{\delta-1}} . \end{equation} 
	Now, since $\delta < e^2 d$ and $\delta>1$, if we choose $\alpha =  \left( e^2 d  \right)^{e^2 d}$, then the right-hand side of \eqref{eq:master2} is at least $n$ and so \eqref{eq:master} holds for all $k\leq n$. 
	
	Finally, observe that by Markov's inequality and \eqref{eq:master} we have 
	\[\Pr{s_k\!\left(G_n,\delta k\right) \geq (2\alpha)^k} \leq \frac{\Ex{s_k(G_n,\delta k)} }{(2\alpha)^k} \leq \frac{\alpha^k}{(2\alpha)^k}= 2^{-k},   \] giving the result for $\beta= 2 \alpha  =  2\cdot\left( e^2 d  \right)^{e^2 d} $. \end{proof}

We now have all we need to prove the main result of this section. 

\begin{proof}[Proof of Theorem~\ref{thm:Gnmisgood}.] We will show that for any $d$ there exists some $c:=c(d)$ such that for $G_n'\sim G\!\left(n,\tfrac{d}{n}\right)$, 
	\begin{equation}\label{eq:bddforGnp} \Pr{G_n'\text{ is not monotone $(c,\Sc_d,\ln^2)$-tiny} } < \frac{20}{n},\end{equation}   
	the result then follows directly from  \Cref{lem:transfer}. 
	
	We first show that Item \eqref{itm:good1} of \Cref{def:good} holds for the function $t :\mathbb N \rightarrow \mathbb{R}$ given by $t(n)=\ln^2 n$ and the class $\mathcal{S}_d$.
	Since~$\mathcal{S}_d$ includes all graphs on at most $1000^{10(d+1)}$ vertices, we can assume that $n\geq 1000^{10(d+1)}\geq 3$ and consider any $k$-vertex subgraph of $G_n'$ with $k\leq  t(n)=\ln^2 n$.
	Observe that since $n\geq 1000^{10(d+1)}$, $ k \leq \ln^2 n$ and $6\leq \ln 1000 \leq 10$, for any $d\geq 0$ we have 
	
	\[\frac{\ln n }{3\cdot \ln k}  \geq \frac{\ln n }{3\cdot 2\ln \ln n }  \geq \frac{10(d+1)\ln 1000}{3\cdot 2\ln(10(d+1)\ln 1000)}\geq  \frac{10(d+1)}{\ln(100(d+1))} \geq \ln((d+1)e^2),  \]where the second inequality holds since $ \ln n/\ln \ln n $ is increasing on the interval $n \in (e^e, \infty)$. Thus, we satisfy the conditions of~\cref{lem:smallkGnp}. Hence, with probability at least $1 - 15/n$, all subgraphs of size $k\leq t(n)$ have at most $k-1+k/\ln k$ edges, and so are contained in $\mathcal{S}_d$.   
	
	We now show that Item \eqref{itm:good2} of \Cref{def:good} holds.
	By~\cref{rem:connectedsuffices}, we only need to establish an exponential upper bound on the number of connected $k$-vertex subgraphs of $G_n'$ for every $k > t(n)$. We proceed with a case distinction on the number of edges in such a $k$-vertex subgraph. 
	\begin{itemize}
		\item {\it Subgraphs on at most $k-1+k/\ln k$ edges.}	By \cref{lem:sparseissmall}, the number $k$-vertex graphs with at most $k-1+k/\ln k$ edges does not exceed $100^k$. This bound holds deterministically.
		\item {\it Subgraphs on at least $k+k/\ln k$ edges.}
		By \cref{lem:largekGnp} and the union bound there exists a constant $\beta = \beta(d)$ such that 
		
		\[\Pr{\sum_{j = \lceil k+k/\ln k \rceil }^{\binom{k}{2}} s_k\!\left(G_n,j\right) \geq k^2\cdot \beta^k} \leq \sum_{j = \lceil k+k/\ln k \rceil }^{\binom{k}{2}} \Pr{s_k\!\left(G_n,j\right) \geq \beta^k}\leq k^2\cdot 2^{-k}. \] Thus, again by taking the union bound, we have 
		\[
		\Pr{\bigcup_{k=t }^{n} \left\{\sum_{j = \lceil k+k/\ln k \rceil }^{\binom{k}{2}} s_k\!\left(G_n,j\right) \geq k^2\cdot \beta^k\right\}}  \leq  \sum_{k=t }^{n} k^2 \cdot 2^{-k} < n^3\cdot 2^{-t} = n^{-\omega(1)},  
		\] 
		since $t=\omega(\log n)$.
	\end{itemize}
	Thus, by combining both cases, we see that with probability $1-n^{-\omega(1)}$ the number of $k$-vertex
	connected subgraphs of $G_n'$ is at most $100^k + k^2\cdot \beta^k \leq c^k$, for some  $c:=c(d)$. Thus, due to \cref{rem:connectedsuffices},  Item \eqref{itm:good2} of \Cref{def:good} holds with probability $1-n^{-\omega(1)}$.

	Hence, by the union bound over Items \eqref{itm:good1} and \eqref{itm:good2}, there exists a~constant $c:=c(d)$ such that the probability $G_n'$ is not monotone $(c,\Sc_d, \ln^2)$-tiny is less than $\frac{15}{n} +\frac{1}{n^{\omega(1)}} < \frac{20}{n}$.  \end{proof}

 	\subsection{Lower bounds}
 	\label{sec:lower-bound}
 	
 	In this section, we prove a~lower bound on the number of induced subgraphs in a~random graph. 
 	
 	\begin{theorem}\label{thm:lowerbound}
 		Let $\delta>0$ be a~sufficiently small constant, $1\leq d \leq n/2$, and $k\in\left[n/\sqrt{d},\sqrt{\delta}n\right]$.
 		Then, w.h.p.\ $G_n \sim G(n,\tfrac{d}{n})$ satisfies $i_k(G_n)>(\delta (n/k)^{\delta/2})^k$.
 		\label{th:rg_lower}
 	\end{theorem}
 	Observe that if we take $k=\big\lceil n/\sqrt{d}\big\rceil$ then w.h.p.\ $G_n\sim G(n,d/n)$  satisfies $i_k(G_n) = d^{\Omega( k)}$. Consequently, as soon as $d$ is allowed to grow with $n$ then, w.h.p.\ $G_n\sim G(n,d/n)$ is not $c$-tiny for any constant $c$. This proves one of the two limit statements in Theorem~\ref{thm:tinyamalg}   and also suggests a~barrier to applying our approach for constructing monotone tiny classes from graphs in $G(n,d/n)$ when $d$ is growing  with $n$.  
 	
\begin{proof}[Proof of Theorem~\ref{thm:lowerbound}.] Let $p:=d/n$. We may assume that $d\geq  1/\delta $ since otherwise $1/\sqrt{d} > \sqrt{\delta}$, and so the range of $k$ for which the theorem holds is empty and the statement holds vacuously. Fix a~$k$-subset $U\subset[n]$. In order to show that w.h.p.\ $i_k(G_n)>(\delta (n/k)^{\delta/2})^k$, we use the following strategy.
 		First, we prove that a~typical $k$-subset $U$ induces a~subgraph with around $k^2p/2$ edges (\cref{cl:typical_U_edges}), and every small enough subset $V\subset U$ induces a~subgraph with at most $\delta k^2p$ edges (\cref{cl:typical_U_subgraphs}).
 		Therefore, if $G_n$ has a~small number of isomorphism classes among its $k$-vertex subgraphs, it should have many pairwise isomorphic $k$-vertex subgraphs with the above two restrictions on their edges.
 		We then show that this is unlikely since most pairs of $k$-sets have small intersections (\cref{cl:overlapping_sets_U}), and w.h.p.\ two $k$-sets in $G_n$ that have a~small intersection induce non-isomorphic subgraphs as soon as edges in these subgraphs satisfy the above-mentioned restrictions (\cref{cl:isomorphic_graphs_are_close}).
 		
 		Set $\mu:={k\choose 2}p$, which is the expected number of edges in a~$k$-vertex induced graph. Note that $\mu \geq\frac{1-o(1)}{2}\cdot n$, since $k\geq n/\sqrt{d}$.
 		

 		\begin{claim}\label{clm:1}
 			W.h.p. the number of edges in $G_n[U]$ is between $(1-\delta)\mu$ and $(1+\delta)\mu$.
 			\label{cl:typical_U_edges}
 		\end{claim}
 		 
 \begin{pocd}{\ref{clm:1}}
 			Note that $\mathbb{E}|E(G_n[U])|=\mu$.
 			Therefore, by the Chernoff bound (\Cref{lem:Chb}), 
 			\begin{align*}
 				\mathbb{P}\left[\Bigl||E(G_n[U])|-\mu\Bigr|>\delta \mu\right] & \leq
 				2\exp\left[-\frac{\delta^2\mu}{2(1+\delta/3)}\right] \\
 				&\leq 2\exp\left[-(1-o(1))\frac{\delta^2 n}{4(1+\delta/3)}\right]\\&=o(1), 
 			\end{align*}	as claimed.	\end{pocd}

 		\medskip
 		\begin{claim}\label{clm:2}
 			W.h.p. every set $V\subset U$ of size at most $\delta k$ induces at most $2\delta \mu$ edges.
 			\label{cl:typical_U_subgraphs}
 		\end{claim}
 		 
\begin{pocd}{\ref{clm:2}}
 			Let $j\leq \delta k$. For a~set $V\subset U$ of size $j$, the induced subgraph $G_n[V]$ has more than $2\delta \mu$ 
 			edges with probability at most
 			$$
 			2\exp\left[-(1+o(1))\frac{(2\delta\mu-\delta^2 \mu)^2}{2(\delta^2 \mu+(2\delta \mu-\delta^2 \mu)/3)}\right]=
 			2\exp\left[-\mu\frac{3\delta(2-\delta+o(1))^2}{4(1+\delta)}\right]
 			$$ 
 			by the Chernoff bound (\Cref{lem:Chb}). Since there are ${k\choose j}$ ways to choose a~set $V\subset U$ of size $j$, and we can take $\delta>0$ sufficiently small,
 			$$
 			\sum_j {k\choose j} \leq k\left(\frac{e}{\delta}\right)^{\delta k}\leq \exp\left[\delta\sqrt{\delta}\ln\left(\frac{e}{\delta}\right)n(1+o(1))\right]=o\left(e^{\delta n/100}\right),
 			$$
 		thus we get that w.h.p.\ every set $V\subset U$ of size at most $\delta k$ induces at most $2\delta\mu$ edges by the union bound.\end{pocd}
 		
 		\bigskip 
 		
 		Let $\xi$ count the number of sets $U\subset[n]$ of size $k$ such that
 		\begin{itemize}
 			\item[($i$)] every set $V\subset U$ of size at most $\delta k$ induces at most $2\delta\mu$ edges in $G_n$,
 			\item[($ii$)] the number of edges in $G_n[U]$ belongs to $\left(\left(1-\delta\right)\mu,\left(1+\delta\right)\mu\right)$.
 		\end{itemize}
 		By~\cref{cl:typical_U_edges,cl:typical_U_subgraphs}, we get that $\mathbb{E} \xi=(1-o(1)){n\choose k}$. Hence, $\mathbb{E}({n\choose k}-\xi)=o({n\choose k})$.
 		By Markov's inequality, 
 		$$
 		\mathbb{P}\left[{n\choose k}-\xi>\delta{n\choose k}\right] \leq  \frac{o({n\choose k})}{\delta{n\choose k}} = o(1).
 		$$
 		Therefore, w.h.p.\ the following event $\mathcal{C}$ holds: $\xi\geq(1-\delta){n\choose k}$. 
 		
 		Let $\mathcal{B}$ be the event that $i_k(G_n)\leq (\delta (n/k)^{\delta/2})^k$. Clearly, $\mathcal{B}\cap\mathcal{C}$ implies the existence of a~$k$-vertex set $U\subset[n]$ satisfying ($i$) and ($ii$) such that $G_n$ has at least $(1-\delta){n\choose k}(\delta (n/k)^{\delta/2})^{-k}$ induced subgraphs isomorphic to $G_n[U]$.
 		We shall use the following two claims to conclude that the latter event is unlikely.

 		\begin{claim}\label{clm:3}
 			W.h.p. in $G_n$ there are no two $k$-vertex sets $U,U'$ sharing at most $\delta k$ vertices such that properties ($i$) and ($ii$) hold for $U$ and $G_n[U]\cong G_n[U']$.
 			\label{cl:isomorphic_graphs_are_close}
 		\end{claim}
 		
\begin{pocd}{\cref{clm:3}}
 			Fix a~$k$-set $U$ and assume that it has properties ($i$) and ($ii$).
 			Fix another set $U'$ that shares at most $\delta k$ vertices with $U$.
 			Due to properties ($i$) and ($ii$), the probability that $G_n[U]\cong G_n[U']$ is at most $k! \cdot p^{(1-3\delta)\mu}$.
 			Indeed, by property ($ii$), the subset $U \cap U' \subset U$ of size at most $\delta k$ induces at most $2\delta\mu$ edges.
 			As, by property ($i$), $G_n[U]$ has at least $(1-\delta)\mu$ edges, there are at least $(1-3\delta)\mu$ edges that should be found at the right place in $G_n[U' \setminus U]$ after committing to one of the $(k-\delta k)! \leqslant k!$ orderings of $U' \setminus U$.
 			
 			By the union bound, the probability that these sets exist is at most
 			\begin{align*}
 				{n\choose k}^2 k! p^{(1-3\delta)\mu} &\leq
 				\exp\left[2k\ln n-(1-3\delta){k\choose 2}\frac{d}{n}\ln\frac{n}{d}\right]\\
 				&\leq\exp\left[k\left(2\ln n-(1-4\delta)\frac{k}{2}\frac{d}{n}\ln\frac{n}{d}\right)\right]\\
 				&\leq\exp\left[k\left(2\ln n-\frac{1-4\delta}{2}\sqrt{d}\ln\frac{n}{d}\right)\right]\\
 				&=o(1),
 			\end{align*} as claimed.
 		\end{pocd} 
 		
 		\medskip
 		\begin{claim}\label{clm:4}
 			Let $U \subset V(G_n)$ be a~set of size $k$.
 			The number of sets $U' \subset V(G_n)$ of size $k$ sharing at least $\delta k$ vertices with $U$ is less than $\frac{1}{2}{n\choose k}(\delta (n/k)^{\delta/2})^{-k}$.
 			\label{cl:overlapping_sets_U}
 		\end{claim}
 		
\begin{pocd}{\ref{clm:4}}
 			Let $k'=\left\lceil\delta k\right\rceil$. Let us fix $U$ and compute the number of sets $U'$:
 			$$
 			\sum_{j\geq k'}{k\choose j}{n-k\choose k-j}\leq\left[\max_{j\geq k'}{n-k\choose k-j}\right]\sum_{j\geq k'}{k\choose j}\leq {n\choose k-k'} 2^k.
 			$$
 			It remains to prove that 
 			\begin{equation}
 				\frac{{n\choose k}}{{n\choose k-k'}}\geq 2\left(2\delta (n/k)^{\delta/2}\right)^k=\exp\left(k\left(\frac{\delta}{2} \ln\frac{n}{k}+\ln(2\delta)\right)+\ln 2\right).
 				\label{eq:binomial_fraction}
 			\end{equation}
 			By applying Stirling's approximation, we obtain
 			
 			\begin{align*}
 				\frac{{n\choose k}}{{n\choose k-k'}}&=
 				\frac{(k-k')!(n-k+k')!}{k!(n-k)!}\\
 				&\sim\frac{\sqrt{2\pi(k-k')}\cdot \sqrt{2\pi(n-k+k')}}{\sqrt{2\pi k }\cdot \sqrt{2\pi(n-k)}}\times\frac{(k-k')^{k-k'}(n-k+k')^{n-k+k'}}{k^k(n-k)^{n-k}}\\
 				&\geq\sqrt{1-\delta}\frac{(n-k+k')^{k'}}{k^{k'}}\left(1+\frac{k'}{k-k'}\right)^{-(k-k')}\left(1+\frac{k'}{n-k}\right)^{n-k}\\
 				&=\frac{(n-k+k')^{k'}}{k^{k'}}\exp\left(-k'+O\left(\frac{k'^2}{k-k'}\right)+k'-O\left(\frac{k'^2}{n-k}\right)\right)\\
 				&=\left(\frac{n-k+k'}{k}\right)^{k'}e^{o(k')}\\
 				&\geq\exp\left[k'\left(\ln\frac{n-k}{k}+o(1)\right)\right]\\
 				&\geq\exp\left[\delta k\left(\ln\frac{n(1-\sqrt{\delta})}{k}+o(1)\right)\right]\\
 				&=
 				\exp\left[\delta k\left(\ln\frac{n}{k}+\ln\left(1-\sqrt{\delta}\right)+o(1)\right)\right],
 			\end{align*}
 			implying \eqref{eq:binomial_fraction} and concluding the proof of the claim.
 		\end{pocd}
 		
 		\bigskip 
 		
We now have what we need to complete the proof of Theorem~\ref{thm:lowerbound}. Let us assume that $\mathbb{P}(\mathcal{B}\cap\mathcal{C})>\delta'$, where $\delta'>0$ is bounded away from 0.
 		From~\cref{cl:isomorphic_graphs_are_close} it follows that w.h.p.\ there are no sets 
 		$U,U'$ satisfying properties ($i$) and ($ii$) sharing at most $\delta k$ vertices and inducing isomorphic subgraphs in $G_n$.
 		But then with probability at least $\delta'-o(1)$ there exists $U$ of size~$k$ satisfying properties ($i$) and ($ii$) and at least $(1-\delta){n\choose k}(\delta (n/k)^{\delta/2})^{-k}-1$ other sets $U'$ having at least $\delta k$ 
 		common vertices with $U$ and inducing subgraphs isomorphic to $G_n[U]$.
 		But the total number of sets $U'$ that have so large intersections with $U$ is smaller by Claim~\ref{cl:overlapping_sets_U}, a~contradiction.\end{proof}
 	
 	\subsection{Threshold for exponentially many unlabeled induced subgraphs}\label{sec:threshold}
	
	For a~graph $G$, we denote by $t(G)$ the size of the largest \emph{homogeneous} (i.e., complete or edgeless) induced subgraph of $G$. Recall that $s(G)$ and $i(G)$ denote the number of unlabeled subgraphs and the number of unlabeled \emph{induced} subgraphs of $G$, respectively.

	A graph $G$ is called {\it $c$-Ramsey}, if $t(G) < \lceil c\log n\rceil$. Note that $c$-Ramsey graphs have order-smallest possible homogeneous induced subgraphs due to the well-known theorem of Ramsey: $t(G)\geq\frac{1}{2}\log n$ for every $G$ (see, e.g.,~\cite{Ramsey}). 
	The conjecture of Erd\H{o}s and R\'{e}nyi, that was resolved by Shelah~\cite{S1998}, says that such graphs have essentially the largest number of unlabeled induced subgraphs possible.
	More precisely, it states that if a~graph $G$ is {\it $c$-Ramsey}, then there exists $\varepsilon>0$ such that $i(G)\geq e^{\varepsilon n}$. This conjecture sparked the study of the relation between the two parameters $t(G)$ and $i(G)$.
	For example, for graphs $G$ with $t(G) < (1-\varepsilon)n$, it is known that $i(G)=\Omega(n^2)$~\cite{AB1989}. Further, in~\cite{EH1989}, it is shown that $$i(G)\geq n^{\Omega\left(\sqrt{n/t(G)}\right)}.$$ This result is far from being tight for graphs with small $t(G)$ as follows from the resolved conjecture of Erd\H{o}s and R\'{e}nyi. In~\cite{AH1991}, it is shown that $i(G)\geq 2^{n/(2t^{20\log(2t)})}$, where $t = t(G)$, which gives a~much better bound for $G$ with $t(G)$ close to $\log n$. 
	However, in general, for non-Ramsey graphs, i.e., graphs $G$ with $t(G) = \omega(\log n)$, tight lower bounds on $i(G)$ are not known.
	
	For a~constant $p$, a~random graph $G_n\sim G(n,p)$ is $2/\log \left(\min\{p,1-p\}^{-1}\right)$-Ramsey w.h.p.\ \cite[Theorem 7.1]{Janson}, and thus, Shelah's result implies an exponential lower bound on $i(G_n)$. For $1/n \leqslant p \leqslant 1/2$, w.h.p.\ $t(G_n)=\Theta (\log(2np)/p)$ (see~\cite[Theorems~7.1,~7.4]{Janson}). In particular, when $p=o(1)$, w.h.p.\ $G_n$ is not $c$-Ramsey for any constant $c>0$. Nevertheless, Theorem~\ref{th:rg_lower} implies that for a~large enough constant $C$ and $C/n \leqslant p \leqslant 1/2$, w.h.p.\ $G_n$ still has exponentially many unlabeled induced subgraphs. Indeed, by taking $k=\lfloor\delta^{4/\delta} n\rfloor$, we deduce from Theorem~\ref{th:rg_lower} that w.h.p.\ $i(G_n) > i_k(G_n)>(1/\delta)^k>e^{\varepsilon n}$ for an appropriate choice of $\varepsilon>0$. Note that, due to symmetry reasons, the upper bound on $p$ does not cause any loss in generality---this bound on $i(G_n)$ can be immediately extended to every $C/n \leqslant p \leqslant 1-C/n$.
	This result is fairly tight as stated below.

\begin{customthm}{1.4}
	\threshold
\end{customthm}

	\begin{proof}  
		Item $(\ref{itm:2})$ (for $p \leq 1/2$) follows directly from Theorem~\ref{thm:lowerbound} by taking $k=\lfloor \delta^{4/\delta} n\rfloor$, however we must prove Item~$(\ref{itm:1})$. Thus, we only need to prove that w.h.p.\ $s(G_n)=2^{o(n)}$ whenever $p<\frac{1-\delta}{n}$. 
		To this end we will use the fact that, within this regime, w.h.p.\ $G_n$ is a~union of connected components of size at most $A \log n$, for some constant $A>0$, and all components contain at most one cycle (see, e.g.,~\cite[Theorems~5.4,~5.5]{Janson}).
		First, note that by Markov's inequality, w.h.p.\ 
		the number of vertices that belong to components of size at least $\sqrt{\log\log n}$ is at most $f_n$ for a~certain $f_n=o(n)$. Indeed,
		letting $X_k$ be the number of vertices that belong to components of size $k$, we get 
		\begin{align*}
			\sum_{k=\sqrt{\log\log n}}^{A\log n}\mathbb{E}[X_k] & \leq 
			\sum_{k=\sqrt{\log\log n}}^{A\log n} k{n\choose k} k^{k-2} p^{k-1}(1-p)^{k(n-k)}\\
			&\leq (1+o(1))\sum_{k=\sqrt{\log\log n}}^{A\log n} (enp(1-p)^n)^k/p \\
			&\leq
			(1+o(1))\sum_{k=\sqrt{\log\log n}}^{\infty} \left(npe^{1-np}\right)^k/p\\
			&=\frac{1+o(1)}{1-npe^{1-np}}\cdot \frac{\left(npe^{1-np}\right)^{\sqrt{\log\log n}}}{np}n\\
			&\leq\frac{e+o(1)}{1-(1-\delta) e^{\delta}} \left(npe^{1-np}\right)^{\sqrt{\log\log n}-1}n\\
			&=o(n),
		\end{align*}
		since $npe^{1-np}$ increases with $p$ in the range, and thus it is at most 
		$(1-\delta) e^{\delta}$.
		
		Now, let $S \subset V(G_n)$ be the set of vertices in the union of all components of size at least $\sqrt{\log\log n}$. We shall bound the number of unlabeled subgraphs of $G_n[S]$ and the number of unlabeled subgraphs of $G_n$ that are entirely outside $S$, and then the total number unlabeled subgraphs in the entire graph is at most the product of these two bounds, since there are no edges between $S$ and $V(G_n)\setminus S$ in $G_n$.
		
		\begin{itemize}
			\item {\it Subgraphs of $G_n[S]$.} There are at most $2^{f_n}$ ways to choose a~subset $U$ from $S$ and there are $2^{o(|U|)}=2^{o(f_n)}$ partitions of $|U|$. For a~fixed $U\subset S$ and its fixed partition,  there are at most $100^{|U|}\leq 100^{f_n}$ unlabeled graphs with connected components whose sizes correspond to a~fixed partition of $|U|$ and with at most one cycle in each component, due to Lemma~\ref{lem:sparseissmall}.
			
			\item {\it Subgraphs of $G_n\setminus S$.} Connected subgraphs of $G_n[V(G_n) \setminus S]$ fall into at most 
			$$
			2^{{\lfloor\sqrt{\log\log n}\rfloor\choose 2}}<2^{\log\log n}=\log n
			$$
			isomorphism classes. Any subgraph of $G_n[V(G_n) \setminus S]$ is a~disjoint union of at most $n$ connected graphs from these isomorphism classes. Thus, the number of unlabeled graphs that are entirely outside $S$ is at most
			$$
			n{n+\lceil\log n\rceil\choose\lceil\log n\rceil}\leq 2^{\log^2 n}.
			$$ 
		\end{itemize}
		Eventually we get that w.h.p.\ 
		$
		s(G_n)\leq 2^{(1+o(1))f_n}\cdot 100^{f_n}\cdot 2^{\log^2 n}=2^{o(n)}.
		$
	\end{proof}

 	\section{Monotone tiny classes can be unwieldy}
 	\label{sec:monotone-tiny}
 
 The aim of this section is to construct a~family of classes that are tiny (thus also \emph{small}) and monotone, but are \textit{unwieldy} in the sense that they cannot be encoded by uniformly bounded labeling schemes (Theorem~\ref{thm:bdddeglabelings}) or described by a~countable collection of $c$-factorial (in particular, small) classes (Theorem~\ref{th:class-complexity}). 
 
 \subsection{Construction of monotone tiny classes}
 \label{sec:general-framework}

 The following lemma is key to our construction of monotone tiny classes with no universal graph of a~given polynomial size. 
 Let $L \subseteq \Nn$ be an infinite subset of the natural numbers and $t:\mathbb{N}\rightarrow \mathbb{R} $ be an increasing function satisfying $t(n) \leq n $ for all $n\in \Nn$.
 We say that $L$ is \emph{$t$-sparse} if there is no pair $x \neq y \in L$ such that $t(y)\leq x \leq y$.   
 
 \begin{lemma}\label{lem:good-tiny} 
 	Let $L\subseteq \mathbb{N}$ be a~$t$-sparse set and let $\gamma$ be a~constant.
 	For every $\ell \in L$, let $\Mc_\ell$ be a~set of $\ell$-vertex monotone $(c,\Yc,t)$-tiny graphs
 	with $|\Mc_\ell| \leq \gamma^{t(\ell)}$. Then $\Zc := \mon(\cup_{\ell \in L} \Mc_\ell)$ is a~monotone tiny class.
 \end{lemma}
 
 \begin{proof}
 	Let $\beta$ be such that $\Yc$ has at most $\beta^n$ unlabeled $n$-vertex graphs for every $n \in \Nn$.
 	Let $(\ell_i)_{i\in \Nn}$ be the natural ordering on the elements of $L$ (by increasing value). 
 	
 	Let $n$ be an arbitrary fixed natural number and $i$ be the smallest index such that $n \leq \ell_i$.
 	Consider an arbitrary $n$-vertex graph $G \in \Zc$. By construction, $G$ is a subgraph of some $G' \in \Mc_{\ell_k}$ with $k \geq i$.
 	If $k > i$, then from $t$-sparseness we have $n \leq \ell_i \leq t(\ell_k) \leq \ell_k$,
 	and since $G' \in \Mc_{\ell_k}$ is monotone $(c,\Yc,t)$-tiny, we conclude that $G$ belongs to $\Yc$.
 	For the same reason, if $k=i$ and $n \leq t(\ell_i)$, then $G$ is in $\Yc$.
 	Thus, if $G$ is not in $\Yc$, then $k=i$ and $t(\ell_i) \leq n \leq \ell_i$, in which case $G$ is one of at most $c^n$ $n$-vertex subgraphs of $G' \in \Mc_{\ell_i}$. Consequently, there are at most $|\Mc_{\ell_i}|\cdot c^n \leq \gamma^{t(\ell_i)} \cdot c^n \leq \gamma^n \cdot c^n$ graphs on $n$ vertices in $\Zc$ that are not in $\Yc$.
 	%
 	
 	Combining these observations, we conclude that the number of unlabeled $n$-vertex graphs in $\Zc$ is bounded from above by
 	\[
 	\beta^n + \gamma^n \cdot c^n \leq \alpha^n,
 	\] 
 	for some constant $\alpha$.   
 \end{proof}

	\subsection{Lower bound on labeling schemes}\label{subsec:lbls}
	
	Let $d \in \Nn$, and $c := c(d)$ be the constant given by Theorem~\ref{thm:Gnmisgood}, $\Sc_d$ be the monotone tiny class given by \Cref{def:classY}, and $t(n) := \ln^2 n$. We denote by $\Xc_{d}$ the class of monotone $(c,\Sc_d,t)$-tiny unlabeled graphs $G$ with $\left\lceil d \cdot(|V(G)|-1)/2\right\rceil $ edges.
	
	\begin{lemma}\label{lem:many-good-graphs}
		For any $d$, the number of unlabeled $n$-vertex graphs in $\Xc_{d}$ is at least $n^{\frac{(d-2)n}{2}(1-o(1))}$.
	\end{lemma}
	\begin{proof}
		The number of labeled graphs in the support of $G\left(n,\left\lceil d(n-1)/2\right\rceil \right)$ is \[\binom{\binom{n}{2}}{\left\lceil d(n-1)/2 \right\rceil }\geq\left(\frac{n}{d}\right)^{\tfrac{d(n-1)}{2}}. \]By Theorem~\ref{thm:Gnmisgood}, a~$1-200\cdot  \sqrt{d/n}$ fraction of these labeled graphs are monotone $(c,\mathcal{S}_d,\ln^2)$-tiny.
		Furthermore, there are at most $n!\leq n^n$ labelings of a~given unlabeled graph. Thus, the number of unlabeled $n$-vertex graphs in $\Xc_{d}$ is bounded from below by
		\[  \bigg(1-200 \cdot \sqrt{\frac{d}{n}}\bigg)\cdot \frac{1}{n^n}\cdot  \left(\frac{n}{d}\right)^{\tfrac{d(n-1)}{2}} = n^{\frac{(d-2)n}{2}(1-o(1))},\]since $d$ is a~constant.
	\end{proof}
	
	We can now show the main result of the paper.
	We recall it for convenience.
	
   \begin{customthm}{1.1}
	\mainthm
\end{customthm}
	
	\begin{proof}
		Suppose, towards a~contradiction, that any monotone tiny class admits universal graphs of size $u_n := n^s$.
		Let $k_n := n^{2s-1}$ and let $d$ be any fixed integer satisfying $\frac{d-2}{2} > s$, say, $d := 2s+4$.
		
		The number of distinct $u_n$-vertex graphs is at most $2^{u_n^2}$ and 
		the number of $n$-vertex induced subgraphs of a~fixed $u_n$-vertex graph 
		is at most $\binom{u_n}{n}$.
		Hence the number of collections of $k_n$~graphs on $n$ vertices that are
		induced subgraphs of a~$u_n$-vertex (universal) graph is at most 
		\begin{equation}\label{eq:cols} 
			2^{u_n^2}\cdot \binom{\binom{u_n}{n}}{k_n}\leq  2^{u_n^2}\cdot u_n^{k_n \cdot n}.  
		\end{equation}

		On the other hand, the number of different collections of $n$-vertex graphs from $\Xc_{d}$ of cardinality $k_n$ is at least 
		\begin{equation}\label{eq:unlabbelled}
			n^{k_n\cdot \frac{(d-2)n}{2}(1-o(1))}.
		\end{equation} 
		
		By taking logarithms, we can see
		that for sufficiently large $n$ the upper bound \eqref{eq:cols} is smaller than the lower bound \eqref{eq:unlabbelled}. Indeed, 
		
		\begin{align*}
			\log\left(2^{u_n^2}\cdot u_n^{k_n \cdot n} \right) 
			& = u_n^2 + k_n\cdot n\cdot  \log u_n \notag \\
			&= n^{2s} + n^{2s}\cdot  s \log n \notag \\
			& = n^{2s}(1 + s \log n),
		\end{align*} 
		while
		\begin{align*} 
			\log\left(n^{k_n\cdot \frac{(d-2)n}{2}(1-o(1))} \right) 
			& = 
			k_n\cdot \frac{(d-2)n}{2}(1-o(1)) \cdot  \log n  \\
			& = n^{2s}\cdot (s+1)(1-o(1)) \cdot  \log n.
		\end{align*}  
		
		Hence, for any sufficiently large $n$ there exists a~collection of $k_n$ unlabeled $n$-vertex graphs from $\Xc_{d}$ not all of which are induced subgraphs of a~$u_n$-vertex graph.	 
		
		We finally seek to apply~\cref{lem:good-tiny}.
		Let $L\subseteq \Nn$ be defined as the image of the function $\ell:\Nn\rightarrow \Nn$, given by $\ell(n+1) = \lceil \exp(\sqrt{\ell(n)}) \rceil$ and $\ell(0)=1$.
		Since, for the function $t=\ln^2$, we have $t(\ell(n)) \leq \ell(n)$ and $\ell(n) \leq t(\ell(n+1)) = (\ln(\ell(n+1)))^2\leq\ell(n+1)$ for all $n\in \Nn$, we see that $L$ is $\ln^2$-sparse. For each $\ell\in L$, let $\Mc_{\ell}$ be a~set of size $k_{\ell}$ of $\ell$-vertex graphs from $\Xc_d$  that is not representable by any $u_{\ell}$-vertex graph, if such a~set exists, and an arbitrary set of size at most~$k_{\ell}$ of $u_{\ell}$-vertex graphs from $\Xc_d$, otherwise.
		Observe that $|\Mc_{\ell}|=k_{\ell}  = \ell^{2s-1} = e^{(2s-1)\ln \ell } \leq \gamma^{\ln^2 \ell} =  \gamma^{t(\ell)}$, for some constant~$\gamma$.
		Thus, by~\cref{lem:good-tiny}, $\Cc=\mon(\cup_{\ell \in L} \Mc_\ell)$ is monotone tiny.
		However, by construction, $\Cc_\ell$ does not admit universal graphs of size $u_\ell = \ell^s$.
	\end{proof}
	
	\subsection{Complexity of tiny classes}
	\label{sec:param-complexity}

A natural approach to describe some family of graph classes is to establish that each class in
the family is a~subclass of a~class from some other family possessing some desirable properties.
For example, the Small conjecture \cite{BGK22} suggested such a~description for the family of small classes.
The conjecture stated that any small class has bounded twin-width, i.e., for any small class~$\mathcal{C}$ there exists a~$k \in \Nn$ such that $\mathcal{C}$ is a~subclass
of the class of graphs of twin-width at most $k$.

Chandoo \cite{Chandoo23} observed that the proof of Hatami and Hatami \cite{HH22} implies that the family of 
factorial classes cannot be described by a~countable set of factorial graph classes. In particular,
the family of factorial classes cannot be described by a~countable set of graph parameters that can
be bounded only for factorial classes. 
We show below that our proof implies similar conclusions for tiny classes. For a~constant $c$, we say that a~hereditary graph class $\mathcal{C}$ is $c$-factorial if, for every $n \in \Nn$, the number of labeled $n$-vertex graphs in $\Cc$ is at most $n^{c n}$. Thus, a~class is factorial if it is $c$-factorial for some $c$.

\begin{customthm}{1.2}	
	\classthmbuff
\end{customthm}

\begin{proof}
	Let $\Fc=\{\mathcal{C}_i\}_{i\in \mathbb{N}}$ be any countable set of $c$-factorial graph classes.
	If we set $C=2c+4$, then, by \cref{lem:many-good-graphs}, we have that the number of $n$-vertex unlabeled graphs in $\Xc_{C}$ is at least $n^{ (c+1) n(1-o(1))} = \omega(n^{ cn})$, and so $\Xc_{C}$ is not a~$c$-factorial class.
	It follows that for every $i\in \mathbb{N}$ there exists a~constant $k_i$ such that for any $k \geq k_i$ there exists a~\mbox{$k$-vertex} graph that belongs to $\Xc_{C}$, but not to $\mathcal{C}_i$.

	Let $t:\mathbb{N}\rightarrow \mathbb{R}$ be an increasing function satisfying $t(n) \leq n $ for all $n\in \Nn$, and let $L \subseteq \Nn$ a~$t$-sparse set.
	We define $r_1$ to be the minimum element in $L$ that is greater than $k_1$, and, for every $j \geq 2$, we define $r_j$ to be the minimum element in $L$ that is greater than $\max\{ r_{j-1}, k_j \}$.
	For every $i \in \Nn$, let $G_i$ be an $r_i$-vertex graph in $\Xc_C$ that does not belong to $\Cc_i$. By the definition of $r_i$ and the above discussion, such a~graph always exists.
	
	Now, by~\cref{lem:good-tiny}, the class $\Zc := \mon(\{G_{i}~:~i \in \Nn\})$ is a~monotone tiny class since we are including at most one $\ell$-vertex graph for every $\ell \in L$.
	Clearly, by construction, $\mathcal{Z}$ is not contained in any class from $\Fc$.
\end{proof}

\begin{definition}\label{def:param}	
	A \emph{graph parameter} is a~function $\sigma$ that assigns to every graph a~number such that
	$\sigma(G_1) = \sigma(G_2)$ whenever $G_1 \cong G_2$.
	The graph parameter $\sigma$ is \emph{bounded} for a~class of graphs $\Cc$
	if there exists a~$\kappa \in \Nn$ such that $\sigma(G) \leq \kappa$ for every $G \in \Cc$; otherwise $\sigma$ is \emph{unbounded} for $\Cc$.
	For a~constant $c>0$, the graph parameter is \emph{$c$-factorial} (resp.\ \emph{small}), if for any fixed number $\kappa$ the class of all graphs $G$ with $\sigma(G) \leq \kappa$ is $c$-factorial (resp.\ small) and the parameter is unbounded for any class that is not $c$-factorial (resp.\ small). 
\end{definition}

We say that a~set $\Sigma$ of graph parameters \textit{describes} a~family $\Fc$ of graph classes if for every graph class $\Cc\in \Fc$, there exists a~parameter $\sigma \in \Sigma$ such that $\sigma$ is bounded for $\Cc$.

\begin{corollary}\label{th:param-complexity}	
	For any $c>0$, and any countable set of $c$-factorial parameters, there exists a~monotone tiny graph class for which every parameter in the set is unbounded.
\end{corollary}

\begin{proof}
	Given any countable set of graph parameters $(\sigma_n)_{n\in \Nn}$, we define the countable family of classes $\Fc=\{\Cc_{n,i} : (n,i)\in \Nn\times\Nn\}$, where $\Cc_{n,i}$ is the class of graphs $G$ with $\sigma_n(G)\leq i$. By Theorem~\ref{th:class-complexity}, there exists a~monotone tiny class $\mathcal{C}$ that is contained in none of the classes in $\Fc$. This implies that every parameter in $(\sigma_n)_{n\in \Nn}$ is unbounded for $\Cc$. Indeed, otherwise we would have $\mathcal{C} \subseteq \Cc_{n,i}$ for some $(n,i)\in \Nn\times\Nn$, a~contradiction. 	\end{proof}

\cref{th:param-complexity} shows that similar attempts to describe the family of small or even the family of monotone tiny classes via a~single or countably many $c$-factorial parameters will fail.
A very special corollary of the above theorem is the negative answer to the Small conjecture \cite{BGK22}, which posits that every small class has bounded twin-width.

      \begin{customcor}{1.1}
	\twwthm
\end{customcor}

\begin{proof}
	This follows from \Cref{th:param-complexity} since twin-width is a~small parameter \cite[Theorem 2.5]{BGK22}, and in particular it is a~$2$-factorial parameter. 
\end{proof}

We note that this conjecture was already refuted \cite{BGTT22}, however the proof relied on some powerful group-theoretic machinery in the form of Osajda's result on so-called \emph{small cancellation labelings}~\cite{Osajda,EG23+}.
One benefit of our proof is that it is self-contained, not relying on this result, nor anything else from group theory.

We also note that Theorem~\ref{th:class-complexity} could alternatively be derived (with some work) from the previous refutation of the Small conjecture \cite{BGTT22}.
Again, the strength of our approach is that it is self-contained and relatively elementary. 
The limit of the construction of~\cite{BGTT22} is that one cannot add many graphs at a~same level. 
Hence it cannot be used to attempt showing~Theorem~\ref{thm:bdddeglabelings}.
	
	\section{Discussion and open problems}
	 \label{sec:discussion}
 
 Our main result shows that for any constant $s$ there exists a~small, even monotone tiny, class of graphs
 that does not have a~universal graph of size $n^s$, or, equivalently, does not admit an adjacency labeling scheme of size $s \log n$. 
 %
 However, the Implicit Graph Conjecture for small (and even for monotone tiny) classes  remains a~challenging open problem.

 \begin{question}
 	\label{q:small-igc}
 	Is there a~small class with no $f(n)$-bit labeling scheme for some $f(n) = \omega(\log n)$? 
 \end{question}
 
 \noindent
 We note that if the Implicit Graph Conjecture fails for monotone tiny classes, it does so by at most a factor of $\log n$ \cite{BDSZZ23mono}.
 
 While it is not known whether every tiny class admits an $O(\log n)$-bit labeling scheme or not,
 there are families of tiny classes (namely, classes of bounded clique-width and classes of bounded twin-width) that do admit such labeling schemes, but no uniformly bounded labeling schemes are currently known. 
 
 \begin{question}\label{q:cw-tww-ls}
 	Is there a~constant $c$ such that every graph class of bounded clique-width (twin-width) admits a~labeling scheme of size $(c+o(1)) \log n$?
 \end{question}	
 
 \noindent We note that, according to the discussion in \cref{sec:param-complexity}, there is qualitative difference between the entire family of tiny classes and the families of classes of bounded clique-width and classes bounded twin-width. Indeed, each of the latter two families can be described by a~single small parameter (clique-width and twin-width, respectively), while the entire family of tiny classes cannot be described even by countably many such parameters. 
 
 %
 	%

 As mentioned in the introduction, ``tiny $\overset{{\scriptscriptstyle ?}}{=}$ small'' is still open. Equality was conjectured in the first arXiv version of~\cite{BNO21} (see Conjecture 8.1, $3. \Rightarrow 2.$) however this conjecture has disappeared from the current arXiv version of \cite{BNO21} (due to being partially refuted, as it contained the Small conjecture). We believe ``tiny $\overset{{\scriptscriptstyle ?}}{=} $ small'' is an interesting question, so we restate it here.
 \begin{question}
 	Is every small class tiny?  
 \end{question}
 
 Finally, we mention an intriguing question about random graphs that arose in our study.
 We found a~threshold for the property of having exponentially many unlabeled induced subgraphs. We do not know whether there exists a~sharp threshold for this property. 
 \begin{question}
 	Let $G_n\sim G(n,p)$. Does there exist $c\geq 1$ such that for any constant $\delta > 0$, if $p<\frac{c-\delta}{n}$, then w.h.p.\ the number of unlabeled induced subgraphs of $G_n$ is $2^{o(n)}$, and, if $p>\frac{c+\delta}{n}$, then w.h.p.\ this number is~$2^{\Theta(n)}$? 
 \end{question}

	\textbf{Acknowledgments.}
        We are grateful to Sebastian Wild for useful discussions on the topic of this paper. A part of this work was done during the 1st workshop on Twin-width, which was partially financed by the grant ANR ESIGMA (ANR-17-CE23-0010) of the French National Research Agency.
        This work has been supported by Research England funding to enhance research culture, by the Royal Society (IES\textbackslash R1\textbackslash 231083), by the ANR projects TWIN-WIDTH (ANR-21-CE48-0014) and Digraphs (ANR-19-CE48-0013), and also the EPSRC project EP/T004878/1: Multilayer Algorithmics to Leverage Graph Structure.

	\bibliographystyle{alpha}
	\bibliography{biblio}
 
\end{document}